\documentclass[reqno]{amsart}
\usepackage{amsfonts, amsmath, amsthm, amssymb, latexsym, graphicx}

\theoremstyle{plain}
\newtheorem{theorem}{Theorem}[section]
\newtheorem{corollary}[theorem]{Corollary}
\newtheorem{lemma}[theorem]{Lemma}

\theoremstyle{definition}
\newtheorem{definition}[theorem]{Definition}
\theoremstyle{remark}
\newtheorem{remark}[theorem]{Remark}
\newtheorem{remarks}[theorem]{Remarks}

\newtheorem{example}[theorem]{Example}
\numberwithin{equation}{section}

\title[Riesz distributions and Laplace transform in the Dunkl setting]
{Riesz distributions and Laplace transform in the Dunkl setting of type A}
\author{Margit R\"osler} 
\address{Insitut f\"ur Mathematik, Universit\"at Paderborn, Warburger Str. 100, D-33098 Paderborn, Germany}
\email{roesler@math.upb.de}

\begin{document}
\date{\today}

\begin{abstract} We study Riesz distributions in the framework of rational Dunkl theory associated with root systems of type A.
As an important tool, we employ a Laplace transform involving the associated Dunkl kernel, which essentially goes back to Macdonald
 \cite{M1}, but was so far only established at a formal level.
We give a rigorous treatment of this transform based on suitable estimates of the type A Dunkl kernel. 
Our main result is a precise analogue in the Dunkl setting of a well-known result by Gindikin, stating that
a Riesz distribution on a symmetric cone is  a positive measure if and only if 
its exponent is contained in the Wallach set. For Riesz distributions in the Dunkl setting, we obtain an analogous characterization in terms of 
 a generalized Wallach set which depends on  the multiplicity parameter on the root system.

\end{abstract}

\maketitle

\section{Introduction}

Riesz distributions play a prominent role in the harmonic analysis on symmetric cones and the study of the wave equation, 
but also 
 in multivariate statistics due to 
their close relation to Wishart distributions, see \cite{FK} for some important aspects. To motivate our results, let us describe a typical example:
Consider the  set $\Omega_n$ of positive definite 
$n\times n$-matrices over $\mathbb R$, which is an open (and actually symmetric) cone in the  space
$ \,\mathrm{Sym}_n = \{ x\in M_n(\mathbb R): x=x^t\}.$ The latter is a Euclidean space (actually, a Euclidean Jordan algebra) with scalar product 
$( x\vert y) = {\rm tr}(xy).$
For indices $\mu \in \mathbb C$ with $ \text{Re}\, \mu > \mu_0:= (n-1)/2$, the Riesz distributions 
$R_\mu$ associated with $\Omega_n$ are defined as the complex Radon measures on
$\mathrm{Sym}_n$ which are defined by 
$$ R_\mu(\varphi) = \frac{1}{\Gamma_{\Omega_n}(\mu)} \int_{\Omega_n}\varphi(x) (\det x)^{\mu -\mu_0-1}  dx $$
where $\Gamma_{\Omega_n}$ is the Gindikin gamma function 
$$ \Gamma_{\Omega_n}(\mu) = \int_{\Omega_n} e^{-{\rm tr}\, x}(\det x)^{\mu -\mu_0-1} dx.$$ 
Considered as tempered distributions on $\mathrm{Sym}_n$, the Riesz measures $R_\mu$ satisfy the recursion
$$ \det\Big(\frac{\partial}{\partial x}\Big) R_\mu  = R_{\mu-1},$$
see \cite{FK}. 
Thus the mappping $\mu \mapsto R_\mu$ uniquely extends to a holomorphic mapping on $\mathbb C$ with values in the space of tempered distributions
$\mathcal S^\prime(\mathrm{Sym}_n).$
Note that for $n=1$, the Riesz distributions are just the homogeneous distributions on $\mathbb R_+ =]0, \infty[$ 
which are obtained by holomorphic extension from the Riemann-Liouville measures 
$$ R_\mu(\varphi) = \frac{1}{\Gamma(\mu)} \int_0^\infty \varphi(x)x^{\mu-1}dx \quad (\text{Re}\, \mu > 0).$$

It is a famous result due to Gindikin \cite{G} that a Riesz distribution associated with a symmetric cone is actually a positive measure if and only if 
its index $\mu$ belongs to the so-called Wallach set. The Wallach set plays an important role in the study of Hilbert spaces of holomorphic functions on symmetric domains, 
see \cite[Chapter XIII]{FK}. 
In the case of the symmetric cone  $\Omega_n$, it is given by
$$ \Bigl\{0, \, \frac{1}{2}\,,\ldots, \frac{n-1}{2}\Bigr\}\cup \Bigl] \frac{n-1}{2}, \infty \Bigr[.$$

In the present paper, we study Riesz distributions in the framework of Dunkl operator theory associated with the  root system 
\begin{equation}\label{A} A_{n-1}= \{ \pm (e_i-e_j): 1 \leq i < j \leq n\} \subset \mathbb R^n. \end{equation}  
(Rational) Dunkl operators are commuting differential-reflection operators associated with a root 
system on some Euclidean space which were intoduced by C.F. Dunkl in \cite{D1}. There is a well-developed harmonic analysis associated with these operators which generalizes both the classical 
Euclidean Fourier analysis as well as the radial harmonic analysis on Riemannian symmetric spaces of Euclidean type. For a general background  see e.g. 
\cite{DX, dJ, R3}. Among the more recent results 
in harmonic analysis associated with Dunkl operators let us mention \cite{ADH, Ve}. 

For $A_{n-1}$, the Dunkl operators in the directions of the standard basis $(e_i)_{1\leq i \leq n}$ of $\mathbb R^n$ are given by 
\begin{equation}\label{Dops} T_i(k) = \frac{\partial}{\partial x_i}  + \,
k\cdot\! \sum_{j\not=i}\frac{1}{x_i-x_j}(1-\sigma_{ij}) \end{equation}
where $\sigma_{ij}$ is the reflection in $\mathbb R^n$ which acts on functions by exchanging the coordinates $x_i$ and $x_j$, 
and $k\in \mathbb C$ is a so-called multiplicity parameter. 
For some values of $k$, Dunkl theory of type $A_{n-1}$  is closely related to the harmonic analysis on symmetric cones, as will be explained in Sections  \ref{Dunkl} and \ref{3}.
For example, analysis on the cone $\Omega_n$ for structures which depend only on the eigenvalues boils down to 
Dunkl analysis of type $A_{n-1}$  for structures on $\mathbb R_+^n$ which are invariant 
under the symmetric group $S_n$; the multiplicity is  hereby  $k= 1/2.$ 
In this paper we shall consider the general, non-symmetric Dunkl setting associated with the root system $A_{n-1}$ 
and arbitrary nonnegative multiplicities. For fixed multiplicity $k\geq 0$, the Riesz measures of type $A_{n-1}$ on $\mathbb R^n$ are defined by 
$$ R_\mu(\varphi) := \frac{1}{d_n(k)\Gamma_n(\mu;k)}\int_{\mathbb R_+^n} \varphi(x) D(x)^{\mu-\mu_0-1}\omega_k(x)dx,\quad \text{Re}\, \mu > \mu_0= k(n-1)$$
where $d_n(k)>0$ is a certain normalization constant, $\Gamma_n(\mu;k)$ is a multivariate version of the Gamma function (see Section \ref{5}), 
\begin{equation}\label{weight} \omega_k(x) = \prod_{1\leq i < j \leq n} |x_i-x_j|^{2k}\end{equation}
is the Dunkl weight function associated with $A_{n-1}$ and multiplicity $k$, and 
$$ D(x) := \prod_{i=1}^n x_i\,.$$
It turns out that the Riesz measures $R_\mu$ satisfy the distributional recursion 
$$ D\big(T(k)\big) R_\mu = R_{\mu-1}$$
with the Dunkl operator $D(T(k)):= \prod_{i=1}^n T_i(k),$
and hence the mappping $\mu \mapsto R_\mu$ extends uniquely to a holomorphic mapping on $\mathbb C$ with values in $\mathcal S^\prime(\mathbb R^n).$ 
This was already observed in the recent
thesis \cite{L}, where Dunkl-type Riesz distributions were introduced to study questions related to  Huygens' principle. 
In this  paper, we carry out a more detailed study of these distributions. As in the case of symmetric cones, an important tool will be 
a suitable version of the Laplace transform, given by
\begin{equation}\label{laplace_intro} \mathcal L_kf(z) = \int_{\mathbb R_+^n} f(x) E_k^A(-x,z) \omega_k(x)dx
\end{equation}
where $E_k^A$ denotes the Dunkl kernel associated with $A_{n-1}$ and multiplicity $k.$ 
This is a non-symmetric variant of a Laplace transform which was first introduced on  a purely formal level by Macdonald in his manuscript \cite{M1}
and was further studied for $n=2$ by Yan \cite{Y}. 
The transform \eqref{laplace_intro}  was used by Baker and Forrester \cite{BF3} and by  Sahi and Zhang \cite{SZ}, but 
due to a lack of knowledge about the decay properties  of the Dunkl kernel, convergence issues could not be properly settled. 

We shall give in Section \ref{3} a rigorous treatment of the Laplace transform \eqref{laplace_intro},  based on suitable estimates for the 
Dunkl kernel of type $A$ which were conjectured  in \cite{BF3}.
In particular, we provide a Cauchy-type inversion theorem, which improves the injectivity statements for the 
Laplace transform
in \cite{Y} and \cite{BF3}. Let us mention at this point that in connection with the Laplace transform, 
specific properties of the type $A$ Dunkl kernel are decisive.
In Section \ref{dist},  we extend the  Laplace transform to distributions. 
Section \ref{5} is then devoted to the study of the Riesz distributions $R_\mu$ in the Dunkl setting. We
compute their Laplace transforms and study for which indices they are actually measures. 
Our  main result (Theorem \ref{maintheorem2}) is a precise analogue of 
Gindikin's result for Riesz distributions on symmetric cones: 
The Dunkl type Riesz distribution $R_\mu$ on $\mathbb R^n$ is a positive measure exactly if $\mu$ belongs to the generalized Wallach set 
\begin{equation}\label{wallachintro}   \{0,k,\ldots,k(n-1)\}\,\cup\, ]k(n-1),\infty[\,.\end{equation}
The Riesz distributions associated with the discrete Wallach points $kr, \, 0\leq r \leq n-1$ can be determined explicitely;  they are
supported in the strata of the cone $\overline{\mathbb R_+^n}$, see Theorem \ref{maintheorem}.  The proofs of Theorems 
\ref{maintheorem} and \ref{maintheorem2} are based on analysis for multivariable hypergeometric functions
which are given in terms of Jack polynomial expansions in the sense of \cite{Ka, M1}, combined with methods of Sokal \cite{S} 
and a variant of the Shanbhag principle from \cite{CL}.
We finally mention that the generalized Wallach set \eqref{wallachintro} also plays an interesting role in connection with integral representations of Sonine type between
Bessel functions of type $B_n$ and the positivity of intertwining operators in the $B_n$-case, see \cite{RV2}.


\section{Dunkl theory for root system $A_{n-1}$}\label{Dunkl}

In this section, we provide a brief introduction to the relevant concepts from rational Dunkl theory; for a background the reader is referred to  
\cite{DX, dJ,  O, R3}.  
We shall consider the root system $A_{n-1}$ in the  Euclidean space $\mathbb R^n$ with the usual scalar product 
$\langle x,y\rangle = \sum_{i=1}^n x_iy_i\,,$ which we extend to $\mathbb C^n\times \mathbb C^n$ in a bilinear way. 
The corresponding finite reflection group is the symmetric group $S_{n}$ on $n$ elements. 
 For fixed multiplicity parameter $k$, the associated Dunkl operators $T_i(k)$ (as defined in \eqref{Dops}) commute, c.f. \cite{D1}. Therefore the 
assignment $x_i \mapsto T_i(k)$ extends to a unital algebra homomorphism $$ p\mapsto p(T(k)), \, \mathbb C[\mathbb R^n] \to \text{End}(\mathbb C[\mathbb R^n]).$$

In this paper, we shall always assume that $ k \geq 0.$ Then for each spectral 
parameter $y = (y_1, \ldots, y_n) \in \mathbb C^n$ there exists
a unique analytic function $f =E_k^A(\,.\,,y)$ satisfying
\[ T_i(k) f = y_i f \,\,\,  \text{ for }\, i =1, \ldots n; \quad f(0) = 1,\]
see \cite{DX, O}. The function $E_k^A$ is called the Dunkl kernel of type $A_{n-1}$. It extends to a holomorphic function on $\mathbb C^n\times \mathbb C^n$ with 
$$ E_k^A(x,y) = E_k^A(y,x), \quad  E_k^A(\lambda x,y) = E_k^A(x, \lambda y), \quad E_k^A(\sigma x,\sigma y) = E_k^A(x,y)$$ 
for all $\lambda \in \mathbb C$ and $\sigma\in S_n.$
 According to \cite{R2}, $E_k^A$ has a positive integral representation. 
More precisely, for each $x\in \mathbb R^n$ there exists a unique probability measure $\mu_x^k\in M^1(\mathbb R^n)$ 
such that
\begin{equation}\label{int_repRo} E_k^A(x,z) = \int_{\mathbb R^n} e^{\langle \xi,z\rangle} d\mu_x^k(\xi) 
\quad \text{for all } z\in \mathbb C^n.\end{equation}
The support of $\mu_x$ is contained in $C(x),$ the convex hull of the orbit of $x$ under 
the action of  $S_n.$ 
Notice  that 
\[E_k^A(x,y)> 0 \quad \text{and }\,  |E_k^A(ix, y)|\leq 1 \quad \text{for all } x,y\in \mathbb R^n .\]

\begin{lemma}\label{estim1} For $x\in \mathbb R^n $ and $y,z \in \mathbb C^n,$
\[ |E_k^A(x, y+z)|\leq E_k^A(x, {\rm Re}\, z) \cdot e^{\max_{\sigma\in S_n}\langle \sigma x, \,{\rm Re}\, y\rangle}.\] 
In particular, 
\[ |E_k^A(x,z)|\leq E_k^A(x, {\rm Re}\, z).\]
\end{lemma}

\begin{proof} According to \eqref{int_repRo},
\begin{align*} \vert E_k^A(x, y+z)\vert  &\leq  \int_{C(x)} e^{\langle \xi,\, {\rm Re}\,y + {\rm Re}\,z\rangle} d\mu_x^k(\xi)\notag\\
&\leq e^{\max_{\sigma\in S_n} \langle \sigma x, \,{\rm Re}\, y\rangle}\cdot \int_{C(x)} 
e^{\langle \xi, {\rm Re}\,z\rangle} d\mu_x^k(\xi),\end{align*}
which implies the assertion.
\end{proof}

\begin{remark} The statement of Lemma \ref{estim1} holds in the context of arbitrary root systems: Consider some (reduced, not necessarily crystallographic) root system $R$ in a Euclidean space $(\frak a, \langle\,.\,,\,.\,\rangle )$ with 
associated reflection group $W$ and a multiplicity function $k \geq 0$ on $R$ (i.e. $k:R \to \mathbb C$ is $W$-invariant). Denote by $E_k$ the associated Dunkl 
kernel. Then for each $x\in \frak a$ there exists a unique probability measure $\mu_x^k$ on $\frak a$, supported in the convex hull of the $W$-orbit of $x$, such that $$E_k(x,z) = \int_{\frak a} e^{\langle \xi,z\rangle} d\mu_x^k(\xi) $$
for all $ z\in \frak a_{\mathbb C}$ (the complexification of $\frak a$). Thus by the same argument as above,
\[ |E_k(x, y+z)|\leq E_k(x, {\rm Re}\, z) \cdot e^{\max_{w\in W} \langle wx, \,{\rm Re}\, y\rangle} \quad \forall x\in \frak a, \, y,z\in \frak a_{\mathbb C}.\] 
\end{remark}

We return to the root system $A_{n-1}$ with multiplicity $k\geq 0$. Recall the weight function $\omega_k$ introduced in \eqref{weight}.
The associated Dunkl transform   on $L^1(\mathbb R^n, \omega_k)$ is defined by 
\[ \widehat f^{\,k}(y) = \frac{1}{c_k}\int_{\mathbb R^n} f(x) E_k^A(x,-iy) \omega_k(x)dx\]
with the normalization constant 
$$ c_k = \int_{\mathbb R^n } e^{-|x|^2/2}\omega_k(x)dx.$$
This is the classical Mehta integral, whose value is given by 
\begin{equation}\label{mehtaconst}  c_k = c_{k,n}  = (2\pi)^{n/2} \prod_{j=1}^n \frac{\Gamma(1+jk)}{\Gamma(1+k)},\end{equation}
c.f. \cite{M3}.
The Dunkl transform is a homeomorphism of the Schwartz space $\mathcal S(\mathbb R^n)$ and
there is an $L^1$-inversion theorem, see \cite{dJ}: if $f\in L^1(\mathbb R^n, \omega_k)$ with $\widehat f^{\,k}\in L^1(\mathbb R^n, \omega_k),$ then
$\, f(x) = (\widehat f^{\,k})^{\wedge\, k}(-x) $ for almost all $x\in \mathbb R^n.$
The Dunkl operators  act continuously on $\mathcal S(\mathbb R^n)$ 
and therefore also on the space $\mathcal S^\prime(\mathbb R^n)$ 
of tempered distributions on $\mathbb R^n$, via
\begin{align}\label{Dunkl_dist} \langle T_\xi(k)u, \varphi\rangle := -\langle u, T_\xi(k) \varphi \rangle, \quad u \in\mathcal S^\prime (\mathbb R^n), 
\, \, \varphi \in \mathcal S(\mathbb R^n).\end{align}
Besides the Dunkl kernel $E_k^A$, we shall also need the Bessel function of type $A_{n-1}$, 
\[J_k^A(x,y):= \frac{1}{n!}\sum_{\sigma\in S_n} E_k^A(\sigma x,y)\]
which is symmetric (i.e. $S_n$-invariant) in both arguments.

\begin{remark}
For certain values of $k$, the Bessel function $J_k^A$ has an interpretation in the context of symmetric spaces. 
In fact, consider for one of the (skew) fields $\mathbb F= \mathbb R, \mathbb C, \mathbb H$ the set 
$\, H_n(\mathbb F) :=\{ x \in M_n(\mathbb F): x = \overline x^t\}\,$
of Hermitian $n\times n$-matrices over $\mathbb F. $
The unitary group $U_n(\mathbb F)$ acts on $H_n(\mathbb F)$ by conjugation $\,x \mapsto uxu^{-1},$  and 
$X_n:= U_n(\mathbb F)\ltimes H_n(\mathbb F)/U_n(\mathbb F)$ is a symmetric space of Euclidean type, which can be 
identified with the tangent space of the symmetric cone $\,\Omega_n(\mathbb F) = \{ x\in H_n(\mathbb F): x \text{ positive definite}\}$ in the
point
$I_n$ (the identity matrix). 
It is well known that the spherical functions of  $X_n$, considered as functions of the spectra of matrices from 
$H_n(\mathbb F),$  can be identified with the Bessel functions $J_k^A(\,.\,, z), z\in \mathbb C^n$ with multiplicity $k = d/2$, 
where $d = \text{dim}_\mathbb R \mathbb F \in \{1,2,4\}$. For details see \cite{dJ2} and \cite{RV}.
\end{remark}

It will be important in this paper that for $k>0,$ 
the Bessel function $J_k^A$ has a series expansion in terms of Jack polynomials. 
To describe this as well as some related facts, we have to introduce further notation; 
references are \cite{Sta, Ka, BF1, FW}. 

Let 
$\Lambda_n^+ $ denote the set of partitions $\lambda= (\lambda_1, \lambda_2, \ldots )$ with at most $n$  parts. 
The number of parts of  $\lambda$ is also called its length and denoted by $l(\lambda).$ 
We consider the Jack polynomials $C_\lambda^{\alpha}:= C_\lambda^{(\alpha)}$ in $n$ variables with parameter $\alpha >0$, which are indexed by partitions $\lambda \in \Lambda_n^+$ 
and are normalized such that 
\begin{equation}\label{norm_Jack}
(z_1 + \cdots + z_n)^m = \sum_{|\lambda|=m} C_\lambda^\alpha (z) \qquad \text{for all }m \in \mathbb N_0,
z \in \mathbb C^n,\end{equation}
where $|\lambda|= \lambda_1 + \cdots + \lambda_n$ denotes the weight of $\lambda$. The polynomials $C_\lambda^\alpha$ are symmetric and  
homogeneous of degree $|\lambda|$. 
If $k>0,$ then according to the relations (3.22) and (3.37)  of \cite{BF3}, $J_k^A$ is a $_0F_0$-hypergeometric function in two variables: 
\begin{equation}\label{power-j-a}
 J_k^A(z,w) =\,  \sum_{\lambda\in \Lambda_n^+} \frac{1}{|\lambda|!}\cdot 
\frac{C_\lambda^{\alpha}(z) C_\lambda^{\alpha}(w)}{C_\lambda^{\alpha}({\underline 1})} =: \, _0F_0^\alpha(z,w) \quad \text{with }\, \alpha = 1/k
\end{equation}
where we use the notation $\,\underline 1 := (1, \ldots, 1) \in \mathbb R^n.$ 
It is known from \cite{KS} that 
\begin{equation}\label{Jack_pos}  C_\lambda^{\alpha}(z) = \sum_{ |\mu|= |\lambda|} c_{\lambda,\mu}^\alpha z^\mu 
\end{equation}
 with nonnegative coefficients 
$c_{\lambda,\mu}^\alpha \geq 0.$ Denoting $\|z\|_\infty = \sup_{1\leq i \leq n} |z_i|,\,$ we therefore have
\begin{equation}\label{jack_hilf} |C_\lambda^{\alpha}(z)| \leq C_\lambda^{\alpha}(\underline 1)\cdot \|z\|_\infty^{\,|\lambda|}\end{equation} 
and thus by \eqref{norm_Jack}
\begin{equation}\label{jack_estim}
\sum_{|\lambda|= m}\frac{|C_\lambda^{\alpha}(z) C_\lambda^{\alpha}(w)|}{C_\lambda^{\alpha}({\underline 1})}   \leq  (n\|z\|_\infty\|w\|_\infty)^m.
\end{equation}
This implies that the series \eqref{power-j-a}  converges locally 
uniformly on $\mathbb C^n\times \mathbb C^n.$ 

\begin{remark} An
 alternative proof of the expansion \eqref{power-j-a} is obtained by symmetrization from an analogous 
expansion of the Dunkl kernel in terms of non-symmetric Jack polynomials, see \cite{R1}, Lemma 3.1 and Example 3.6. 

\end{remark}


\section{The type $A$ Laplace transform}\label{3}

We again consider the root system $A_{n-1}$ with 
some fixed multiplicity $k \geq 0.$
For $x=(x_1, \ldots, x_n), \, y=(y_1, \ldots, y_n)\in \mathbb R^n$ we write $x\leq y$ (and $y\geq x$) if $\, x_i \leq y_i$ for all $i=1, \ldots n.$ 
This defines a partial order on $\mathbb R^n$. For $s \in \mathbb C$, 
we use the abbreviation 
$$ \underline s:= (s, \ldots, s) \in \mathbb C^n.$$
We further write $\mathbb R_+ := ]0,\infty[ $ and put
\[ \|z\|_1:= \sum_{i=1}^n |z_i| \quad \text{for }\, z\in \mathbb C^n.\]

The rigorous foundation of the Laplace transform in the Dunkl setting associated with $A_{n-1}$ will be based on the 
following factorization and  exponential decay of the Dunkl kernel $E_k^A$.

\begin{lemma}\label{estim2}
\begin{enumerate}
 \item[\rm{(1)}] For all $x, z\in \mathbb C^ n$ and $s\in \mathbb C$, 
\[ E_k^A(x, z + \underline s) = e^{\langle x,\underline s\,\rangle} \cdot E_k^A(x,z).\]
\item[\rm{(2)}] Let $x\in \overline{\mathbb R_+^n}$ and $a\in \mathbb R^n$. Then for all $z\in \mathbb C^n$ with 
${\rm Re}\, z \geq a$,
\[ |E_k^A(-x,z)|\leq E_k^A(-x, a)\leq \exp\,\bigl(-\|x\|_1\cdot \min_{1\leq i\leq n} a_i\bigr).\]
In particular, if ${\rm Re}\, z \geq \underline s\,$ for some $s>0$, then 
$$ |E_k^A(-x,z)| \leq e^{-s\|x\|_1} \quad \forall\, x\in \overline{\mathbb R_+^n}.$$
\end{enumerate}
Properties (1) and (2) are also valid for the Bessel function $J_k^A$. 
\end{lemma}

\begin{proof} (1) By analyticity, it suffices to consider $x\in \mathbb R^n.$ Then the assertion follows easily from Proposition 3.19 of \cite{BF3}, where $s=1$.  It can also be  deduced from formula
\eqref{int_repRo}, as follows:
We have $\langle \xi, \underline 1\rangle = \langle x, \underline 1\rangle $ for all $\xi$ in the $S_n$-orbit of $x$,  which implies that
$\, \langle \xi, \underline s\rangle = \langle x, \underline s\rangle $ for all $\xi\in C(x)$ and all $s\in \mathbb C$.
The statement is then immediate from formula \eqref{int_repRo}.

(2) In view of Lemma \eqref{estim1} it suffices to consider $z\in \mathbb R^n$ with $z\geq a$.
Our assumption $x\geq 0$ implies that $\xi\geq 0$ for all $\xi \in C(x)$ and therefore also
$\langle \xi, z-a\rangle \geq 0$. 
Thus by \eqref{int_repRo}, 
\[ |E_k^A(-x,z)| = \int_{C(x)} e^{-\langle \xi, z-a\rangle} e^{-\langle \xi, a\rangle} d\mu_x^k(\xi) 
 \, \leq \,\int_{C(x)} e^{-\langle \xi, a\rangle}d\mu_x^k(\xi) = E_k^A(-x, a).\]
For the second inequality, we start with the last equality in the above formula 
and write $\xi\in C(x)$ as \[ \xi = \sum_{\sigma\in S_n} \lambda_\sigma  \sigma x \,\,\text{ with }\,
\,\lambda_\sigma\geq 0, \,\sum_{\sigma\in S_n} \lambda_\sigma = 1.\] 
Using the estimate
\[ \langle \sigma x,a\rangle \geq \|x\|_1 \cdot \min_{1\leq i \leq n} a_i \quad (\sigma\in S_n)\]
we obtain that $\, \langle \xi,a\rangle \geq \|x\|_1 \cdot \min a_i \,.$ This implies statement (2).  The same assertions for $J_k^A$ are immediate.
\end{proof}

Following \cite[Section 3.4]{BF3}, we define the type $A$ Laplace transform   of a function $f\in L_{\textrm{loc}}^1(\mathbb R_+^n,\omega_k)$ as
\begin{equation}\label{Laplacedef}
 \mathcal L_kf(z) := \int_{\mathbb R_+^n} f(x) E_k^A(-x,z) \omega_k(x)dx \quad (z\in \mathbb C^n),\end{equation}
provided the integral exists.

\begin{remarks} 1. In \cite[formula (2) on p.~38]{M1}, Macdonald
defines the Laplace transform with the $S_n$-invariant
kernel $$e(-x,z) = J_k^A(-x,z)$$ instead of $E_k^A(-x,z)$ (see \cite[p.26]{M1} for the definition of  $e$). If $f$ is $S_n$-invariant, then in  \eqref{Laplacedef} the Dunkl kernel $E_k^A(-x,z)$ may be 
replaced by $J_k^A(-x,z)$ 
without affecting the value of the integral. So in the symmetric case,
 our definition coincides with that of Macdonald up to a constant 
factor.

2. Macdonald's definition of the Laplace transform in \cite{M1} is closely related to the well-known Laplace transform on symmetric cones. To explain this relation, 
suppose that $V$ is a simple Euclidean Jordan algebra with Jordan multiplication $(x,y) \mapsto xy$ and scalar product 
$(x\vert y)=  {\rm tr}(xy)$ where ${\rm tr}$ denotes the Jordan trace on $V$, i.e. ${\rm tr}(x)$ is the sum of eigenvalus of $x$.
Let $\Omega \subset V$ be the associated symmetric cone. 
It can be written as a Riemannian symmetric space $\Omega=G/K$ 
where $G$ is the identity component of the automorphism group of $\Omega$ and $K = G\cap O(V).$ We refer to \cite{FK} for these facts and 
a general introduction to the analysis on symmetric cones. 
The Laplace transform of a function $F\in L_{\textrm{loc}}^1(\Omega)$ is defined by 
$$ \mathcal L F(y) = \int_\Omega F(x) e^{-(x\vert y)}dx  \quad (y\in V),$$ 
provided the integral exists. Suppose the rank of $V$ is $n.$ Then the  possible ordered spectra of elements from $\Omega$ are given by the set
$$ C_+ = \{\xi = (\xi_1, \ldots, \xi_n)\in \mathbb R^n: \xi_1 \geq \cdots \geq \xi_n >0\}.$$ 
If $F$ is $K$-invariant, it can be uniquely written as $F(x) = f(\text{spec}(x)),$ where $\text{spec}(x)\in C_+$ denotes the 
set of eigenvalues of $x$ ordered by size and $f: C_+\to 
\mathbb C$ is measurable.
Fix some Jordan frame $(c_1, \ldots, c_n)$ of $V$ (that is, the $c_j$ form a complete system of orthogonal primitive idempotents in $V$). For $\xi \in C_+$
let $\underline \xi: = \sum_{j=1}^n \xi_j c_j\in \Omega.$ 
Then according to Theorem VI.2.3 of \cite{FK}, 
$$ \mathcal LF(y) = c_0\!\! \int_{C_+} f(\xi)\left(\int_K e^{-(k\underline \xi\vert y)}dk \right)\!\prod_{1\leq i<j\leq n} \!(\xi_i-\xi_j)^d\, d\xi$$
where $d\in \mathbb N$ denotes the Peirce constant of $V$ and $c_0>0$ is some normalization constant depending on $V.$ 
In order to identify the  integral over $K$, we recall
the spherical (or zonal) polynomials $Z_\lambda$ on  $V$ which are indexed by partitions $\lambda\in \mathbb N_0^n$ and normalized such that for each $m\in \mathbb N_0$,
\begin{equation}\label{trace_spherical} ({\rm tr }\,x)^m = \sum_{|\lambda|=m} Z_\lambda (x) \quad (x\in V),
\end{equation}
see Section XI.5 of \cite{FK}. 
The $Z_\lambda$ are $K$-invariant and thus depend only on the eigenvalues of their argument. As such, they 
are given by Jack polynomials:
\begin{equation}\label{prod_spherical} Z_\lambda(x) = C_\lambda^{\alpha}\bigl(\text{spec}(x)\bigr) \quad \text{with } \alpha = \frac{2}{d},\end{equation}
see the notes to Chap. XI in \cite{FK}. 
Further, the $Z_\lambda$ satisfy the product formula
$$ \frac{Z_\lambda(x) Z_\lambda(y)}{Z_\lambda(e)} = \int_K Z_\lambda\left(P(\sqrt x\,) ky\right) dk \quad (x\in \Omega, y\in V)$$
where $P$ denotes the quadratic representation of $V$.
This is immediate from \cite[Corollary XI.3.2]{FK} and the fact that $P(\sqrt x)e = x.$

Now consider the type $A$ Bessel function $J_k^A$ with multiplicity $k=d/2$. Let $x\in \Omega, \,y\in V$ and $\xi = \text{spec}(x), \,\eta = \text{spec}(y).$ 
Then by relations \eqref{trace_spherical} and \eqref{prod_spherical},
\begin{align*} J_{d/2}^A(\xi, \eta) &=\, 
\sum_{\lambda\geq 0} \frac{1}{|\lambda|!}\frac{Z_\lambda(x)Z_\lambda(y)}
{Z_\lambda(e)}  \\
&=\,\int_K  e^{{\rm tr}(P(\sqrt x) ky)} dk = \int_K e^{ (x\vert ky)} dk.
\end{align*}
Therefore $\mathcal LF(y)$ depends only on  $\eta = \text{spec}(y)$ and is given by 
$$ \mathcal LF(y) = c_0\!\! \int_{C_+} f(\xi) J_{d/2}^A (-\eta,\xi)\,\omega_{d/2}(\xi)d\xi.$$
Extending $f$ to a symmetric function on $\mathbb R_+^n$, this becomes
$$ \mathcal LF(y) =  \frac{c_0}{n!} \int_{\mathbb R_+^n} f(\xi)J_{d/2}^A (-\xi,\eta)\,\omega_{d/2}(\xi)d\xi,$$
which coincides, up to a constant, with Macdonald's Laplace transform for $k=d/2.$ 
\end{remarks}

Let us now continue the study of the type $A$ Laplace transform $\mathcal L_k\,.$

\begin{lemma}\label{Laplace_1} Let $f \in L^1_{\textrm{loc}}(\mathbb R_+^n)$ and suppose that $\mathcal L_kf(a)$ exists for some $a \in \mathbb R^n$, that is
\[\int_{\mathbb R_+^n} |f(x)| E_k^A(-x,a) \,\omega_k(x)dx < \infty.\]
Then the following hold.
 \ \begin{enumerate}
 \item[\rm{(1)}] $\mathcal L_kf(z)$ exists for all $z\in \mathbb C^n$ with ${\rm Re}\, z\geq a$, and $\mathcal L_kf$ is holomorphic on the half space
 $$ H_n(a) := \{z\in \mathbb C^n: {\rm Re}\, z > a\}.$$
  \item[\rm{(2)}] If $p\in \mathbb C[\mathbb R^n]$ is a polynomial, then $\mathcal L_k(fp)(z)$ exists for all $z\in H_n(a),$ and 
 \[ p(-T(k)) (\mathcal L_kf)  = \mathcal L_k (fp)\quad \text{on }\, H_n(a).\]
 \end{enumerate}
\end{lemma}

 \begin{proof} Part (1) is immediate from Lemma \ref{estim2}(2) and standard theorems for
 holomorphic parameter integrals. For part (2), let $z\in H_n(a)$ and choose $\epsilon >0$ such that $\text{Re}\,z > a + \underline \epsilon. $ 
 Then for $x \in \mathbb R_+^n$ we have  $\,|E_k^A(-x,z)|\leq e^{-\epsilon \|x\|_1} E_k^A(-x,a),$ due to Lemma \ref{estim2}. 
 This implies that $\mathcal L_k(fp)(z)$ exists for each polynomial $p$, and   differentiation under the integral gives
  \[ -T_{e_i}(k)\mathcal L_kf(z) = \int_{\mathbb R_+^n} x_if(x) E_k^A(-x,z)\omega_k(x)dx, \quad 1\leq i \leq n.\]
  The statement now follows by induction.
 \end{proof}
  
\begin{example} Suppose that $f$ is measurable on $\mathbb R_+^n $  and exponentially bounded according to
 \[|f(x)| \leq C e^{s \|x\|_1}\] with constants $C>0$ and $s\in \mathbb R$. Then 
 by Lemma \ref{estim2}(2), $\mathcal L_kf(z)$ exists for all $z\in H_n(\underline s).$
 \end{example}

  We continue with some further elementary properties of the Laplace transform:
 \begin{lemma}Suppose that $f$ is measurable on $\mathbb R_+^n$ with $\,|f(x)|\leq C \cdot e^{s \|x\|_1}$ for some $s \in \mathbb R$. Then
  \begin{enumerate}\itemsep=1pt
  \item[\rm{(1)}] For $z\in H_n(0)$, 
   $ \mathcal L_k\bigl(e^{-\langle x, \underline s\rangle} f\bigr)(z) = \mathcal L_k(f)(z+\underline s).$ 
    \item[\rm{(2)}] Let $y\in \mathbb R^n$. Then $\mathcal L_kf(x+iy) \,\longrightarrow \, 0\, \text{ as }\min x_i \to \infty.$
 \item[\rm{(3)}] Let $x > \underline s$. Then $\mathcal L_kf(x+iy) \,\longrightarrow \, 0\,\, \text{ as } \min y_i \to \infty.$
    \end{enumerate}
\end{lemma}
 
 \begin{proof} (1) This is obvious from Lemma \ref{estim2}(1) (c.f. also \cite{BF3}).
 
  (2)  Write $x = x^\prime + \underline \xi$ with $\xi = \min x_i$ and $x^ \prime \geq 0.$ 
 Then 
 \[\mathcal L_kf(x+iy) = \int_{\mathbb R_+^n} e^{-\langle u, \underline \xi\rangle} E_k^A(-u, x^\prime +iy)f(u) \omega_k(u)du\]
 where
\[
 \big|e^{-\langle u, \underline \xi\rangle} E_k^A(-u, x^\prime +iy)f(u)\big|\, \leq C\cdot 
  e^{(s-\xi)\|u\|_1} E_k^A(-u, x^\prime)\,\leq \,
    C\cdot e^{(s-\xi)\|u\|_1}.
\]
As $\xi\to \infty$, the dominated convergence theorem yields the assertion.

 (3) As above, write $y = y^\prime +\underline \eta$ with $\eta = \min y_i$ and $y^\prime \geq 0.$ Then
 $$ \mathcal L_kf(x+iy) = \int_{\mathbb R_+^n} f(u) E_k^A(-u, x+iy^\prime)
 e^{-i\langle u, \underline\eta\rangle} \omega_k(u)du.$$
 The statement now follows from Lemma 3.1 together with the Riemann-Lebesgue Lemma for the classical Fourier transform.

\end{proof}

Our next result is a Cauchy-type inversion theorem for the Laplace transform.
We extend the weight function $\omega_k$ to $\mathbb C^n$ by 
\[ \omega_k(z):= \prod_{1\leq i<j\leq n} |z_i-z_j|^{2k}.\]

\begin{theorem} Suppose that $\mathcal L_kf(\underline s)$ exists  for some $s\in \mathbb R$, and 
 that $$y\mapsto \mathcal L_kf(\underline s + iy)\in L^1(\mathbb R^n, \omega_k).$$
Then $f$ has a continuous representative $f_0$, and
\begin{equation}\label{inv1} \frac{(-i)^n}{c_k^2} \int_{{\rm Re}\, z =\underline s}  
\mathcal L_kf(z) E_k^A(x,z)\, \omega_k(z) dz\,= \, \begin{cases} f_0(x) &\,\text{if } x > 0\\
    0 & \>\>\text{otherwise.}\end{cases}
\end{equation}
Here $dz$ is understood as an $n$-fold line integral. 

\end{theorem}

\begin{proof} 
Lemma \ref{Laplace_1} assures that $\mathcal L_kf(\underline s+iy)$ indeed exists for all $y\in \mathbb R^n$.
By Lemma \ref{estim2},
the left-hand side of \eqref{inv1} can also be written as
\[\frac{1}{c_k^2}\, e^{\langle x, \underline s\rangle}\int_{\mathbb R^n} \mathcal L_kf(\underline s + iy) E_k^A(x,iy)\, \omega_k(y)dy.\]
This integral is absolutely convergent by our assumption.
Extend $f$ to $\mathbb R^n$ by 
\[ \widetilde f(x):= \begin{cases} f(x) & \text{ if }\, x > 0\\
                       0 & \text{ otherwise}
          \end{cases}\]
                   and put $F(x)= e^{-\langle x, \underline s\rangle} \widetilde f(x).$
As $e^{-\langle x, \underline s\rangle}= E_k^A(-x, \underline s),$ we have $F\in L^1(\mathbb R^n, \omega_k).$ In view of Lemma \ref{estim2}, the Dunkl transform of $F$ is given by
\[ \widehat F^{\,k}(y) = \frac{1}{c_k} \int_{\mathbb R_+^n} e^{-\langle x, \underline s\rangle }f(x) E_k^A(-ix,y)\omega_k(x)dx \,=\, \frac{1}{c_k}\mathcal L_kf(\underline s + iy), \,\, y\in \mathbb R^n.
\]
From this relation, our assumption and the $L^1$-inversion theorem for the Dunkl transform  it follows that $F$ has a continuous representative with
\[ c_k^2 F(x)= c_k\int_{\mathbb R^n} \widehat F^{\, k}(y) E_k^A(ix,y)\omega_k(y)dy \, = \, \int_{\mathbb R^n} \mathcal L_kf(\underline s + iy) E_k^A(iy,x) \omega_k(y)dy.\] 
Hence  $\widetilde f$ has a continuous representative as well, satisfying
 \begin{align*}\widetilde f(x) &= \frac{1}{c_k^2}\int_{\mathbb R^n} \mathcal L_kf(\underline s + iy) E_k^A(x,\underline s +iy) \omega_k(\underline s+iy) dy\\
                               & = \frac{(-i)^n}{c_k^2} \int_{{\rm Re}\, z =\underline s}  \mathcal L_kf(z) 
                               E_k^A(x,z) \omega_k(z) dz.
 \end{align*}
This implies the assertion.
\end{proof}

As an immediate consequence of this theorem, we obtain

\begin{corollary}[Injectivity of the Laplace transform] Let $f\in L_{\textrm{loc}}^1(\mathbb R_+^n,\omega_k)$ and $s\in \mathbb R$ such that 
$\mathcal L_kf(\underline s)$ 
exists and $\,\mathcal L_kf(\underline s+iy) =0\,$ for all $y\in \mathbb R^n.$ Then $f= 0$ a.e. 
\end{corollary}

\begin{remark} Following a method of Yan \cite{Y}, Baker and Forrester \cite{BF3}  proved a  weaker injectivity result 
 for the Laplace transform $\mathcal L_k$ on a certain weighted $L^2$-space by applying suitable Dunkl operators to 
 the Laplace integral.
 \end{remark}

 \section{The type $A$ Laplace transform of tempered distributions}\label{dist}

In this section, we define the  Laplace transform of tempered distributions
in the Dunkl setting.  We follow the classical 
approach, see e.g. \cite{V}. Denote by $\mathcal S(\mathbb R^n)$ 
the classical Schwartz space of rapidly decreasing functions on $\mathbb R^n$ and by $\mathcal S^\prime(\mathbb R^n)$ the space
of tempered distributions on $\mathbb R^n.$ 
Let further 
 $$\mathcal S_+^\prime(\mathbb R^n) = \{u\in \mathcal S^\prime(\mathbb R^n): \textrm{supp}\, u\subseteq \overline{\mathbb R_+^n}\,\}$$
denote the set of tempered distributions supported in $\overline{\mathbb R_+^n}\,$. In order to define the Laplace transform of $u\in \mathcal S_+^\prime(\mathbb R^n),$ choose a 
cutoff  function  $\chi\in C^\infty(\mathbb R^n)$ with $\,\textrm{supp}\, \chi \subseteq ]-\epsilon, \infty[^n $ for some $\epsilon >0$ and $\chi(x) = 1$ in a neighborhood
of $\overline{\mathbb R_+^n}.$ 

\begin{lemma}\label{Dunkl_Schwartz} For each $z\in H_n(0)$ the function $\,x\mapsto \chi(x) E_k^A(x, -z)$ belongs to 
$\mathcal S(\mathbb R^n)$. 
\end{lemma}

\begin{proof}
We use the following estimates from \cite{dJ} for the partial derivatives of the Dunkl kernel:
There are constants $C_\nu>0, \, \nu \in \mathbb N_0^n,$ such that for all $x\in \mathbb R^n$ and $ z\in \mathbb C^n,$
\begin{equation}\label{estimableitung} \vert \partial_x^\nu E_k^A(x,z)\vert 
\leq C_\nu \cdot \|z\|_\infty^{\,|\nu|}\,e^{\max_{\sigma\in S_n}\langle \sigma x, \text{Re}\, z\rangle}.\end{equation}
A short computation shows that 
for $x\in \mathbb R^n$ with $x >-\underline \epsilon,\, z\in \mathbb C^n$ with $\,\text{Re} \,z \geq \underline s >0$ 
and $\sigma\in S_n$, 
$$ \langle \sigma x\,,\text{Re} \,z\rangle  \geq s\cdot\sum_{i=1}^n x_i \,-\, \epsilon\cdot \|\text{Re}\, z - \underline s\|_1\,.$$ 
Therefore
\begin{equation}\label{decay} \vert\partial_x^\nu E_k(x,-z)\vert \leq C_\nu \|z\|_\infty^{\,|\nu|}\,
e^{\,\epsilon \|\text{Re} \, z - \underline s\|_1}\cdot e^{-s\sum_{i=1}^n x_i}.\end{equation}
This easily implies the assertion.
\end{proof}

\begin{definition} The Laplace transform of  $u\in \mathcal S_+^\prime(\mathbb R^n)$ is defined by
 $$ \mathcal L_ku: H_n(0)\to \mathbb C, \,\, \mathcal L_k u(z):= \langle\, u_x,\chi(x) E_k^A(x, -z)\rangle,$$
 where the notation $u_x$ indicates that $u$ acts on functions of the variable $x$, and the cutoff function $\chi$ is as above.  
 As $u$ is supported in $\overline{\mathbb R_+^n}$, this definition  is independent of the choice of $\chi.$  
\end{definition}

\begin{remark}\label{Laplace_measure}
If $m \in \mathcal S_+^\prime(\mathbb R^n)$ is of order zero, i.e. a complex tempered Radon measure supported 
in $\overline{\mathbb R_+^n}$, then its Laplace transform is given by
$$ \mathcal L_km(z) = \int_{\overline{\mathbb R_+^n}} E_k^A(x,-z) dm(x), \quad z\in H_n(0).$$
The exponential decay properties of $E_k^A$ in Lemma \ref{estim2} together with Morera's theorem imply that $\, \mathcal L_km$ is holomorphic 
on $H_n(0)$ and may be differentiated under the integral. 
\end{remark}

\begin{example}\label{ex_delta_0} Denote by $\delta_x$ the Dirac distribution in the point $x\in \mathbb R^n.$ Then
$\mathcal L_k(\delta_0) = 1.$ 
\end{example}

\begin{theorem}[Injectivity of the Laplace transform of tempered distributions] \label{inj_distributions}
Let $u\in \mathcal S_+^\prime(\mathbb R^n)$ and suppose that there is some 
$\,s\in \,]0,\infty[$ such that $\mathcal L_ku(\underline s+iy) = 0$ for all $y\in \mathbb R^n.$ Then $u=0.$ 
\end{theorem}

\begin{proof} Fix a cutoff
 function $\chi $ as above, and let $\varphi \in \mathcal D(\mathbb R^n).$ 
 By Lemma \ref{Dunkl_Schwartz}, the function
 $\xi\mapsto \chi(\xi)E_k(-\underline s-iy, \xi)$ belongs to $\mathcal S(\mathbb R^n)$ for each $y\in \mathbb R^n$, and it is easily checked that 
 $$ \psi(\xi,y) := \chi(\xi)E_k^A(-\underline s-iy, \xi)\varphi(y) \in \mathcal S(\mathbb R^n\times \mathbb R^n).$$
Consider the weight function $\omega_k$ as a regular
tempered distribution on $\mathbb R^n$ in the usual way. 
Then by Fubini's theorem for tensor products of tempered distributions (see e.g. \cite[Section 5.5] {V}),
\begin{align*}
0 & = \, \int_{\mathbb R^n} \mathcal L_ku(\underline s+iy) \varphi(y)\, \omega_k(y) dy \\
& =\,\int_{\mathbb R^n} \big\langle u_\xi, \chi(\xi) E_k^A(-\underline s-iy, \xi)\big\rangle\,\varphi(y)\,\omega_k(y)dy \\
&= \,\int_{\mathbb R^n} \langle u_\xi, \psi(\xi,y)\rangle \,\omega_k(y)dy = \,
\langle u\otimes \omega_k, \psi\rangle \,= \,\big\langle u_\xi, \int_{\mathbb R^n} \psi(\xi, y)\, \omega_k(y)dy\big\rangle\,  \\
 & =\,\big\langle u_\xi, \chi(\xi) e^{-\langle \underline s, \xi\rangle} \int_{\mathbb R^n} E_k^A(-iy,\xi) \varphi(y)\,\omega_k(y)dy\big\rangle\\
 &= \, \big\langle e^{-\langle \underline s,\,\cdot\, \rangle}u\,, \widehat \varphi^{\,k}\big\rangle.  \end{align*}
Since $\mathcal D(\mathbb R^n)$ is dense in $\mathcal S(\mathbb R^n)$ and the Dunkl transform is a homeomorphism of $\mathcal S(\mathbb R^n)$, it follows that  $u=0$. 
\end{proof}

\section{Riesz distributions in the type $A$ Dunkl setting}\label{5}

In this section we assume that  $k>0.$  We put 
$\, \mu_0:=k(n-1)\,$ and introduce the normalization constant 
$$ d_n(k):= \prod_{j=1}^n \frac{\Gamma(1+jk)}{\Gamma(1+k)} = (2\pi)^{-n/2}c_{k,n}\,$$
with the Mehta constant $c_{k,n}$ from \eqref{mehtaconst}, 
as well as the multivariable Gamma function and generalized Pochhammer symbol
$$ \Gamma_n(\mu;k):= \prod_{j=1}^n \Gamma\big(\mu-k(j-1)\big), \quad [\mu]_\lambda^{k} := \prod_{j=1}^{l(\lambda)}\big(\mu-k(j-1)\big)_{\lambda_j}\,$$
for $\mu\in \mathbb C$ and partitions $\lambda \in \Lambda^+_n\,.$ Here $(a)_n = a(a+1)\cdots (a+n-1).$ 
Notice that the pole set of $\Gamma_n(\,.\,;k)$ is given by $\{0, k, \ldots, k(n-1)\}- \mathbb N_0\,.$ 

\smallskip
Before turning to Riesz distributions, we provide some Laplace transform formulas which will be useful in  the sequel.  
Recall the Jack polynomials $C_\lambda^\alpha$ in $n$ variables with $\alpha >0.$ It will be convenient to work with the renormalized polynomials
$$ \widetilde C_\lambda^\alpha(z):= \frac{C_\lambda^\alpha(z)}{C_\lambda^\alpha(\underline 1)}, \quad\underline 1 = (1, \ldots, 1) \in \mathbb R^n.$$
Recall our notation $D(x) = \prod_{i=1}^n x_i$ for $x\in \mathbb R^n.$
We need the following integral formula which is due to  Macdonald.

\begin{lemma}[\cite{M1}, formula (6.18)]\label{lemma_Macdonald}   For $\mu \in \mathbb C$ with $\,{\rm Re}\, \mu > \mu_0$ and $\lambda \in \Lambda^+_n$, 
\begin{equation}\label{int_jack_1} 
\int_{\mathbb R_+^n} \widetilde C_\lambda^{\,1/k}(x)e^{-\langle x,\underline 1 \rangle }D(x)^{\mu-\mu_0-1} \omega_k(x)dx = 
  \,d_n(k)\Gamma_n(\mu;k)\cdot [\mu]_\lambda^k\,. 
\end{equation}
\end{lemma}
For $\lambda=0$ this becomes
$$
 \int_{\mathbb R_+^n} e^{-\langle x, \underline 1 \rangle } D(x)^{\mu-\mu_0-1} \omega_k(x)dx = \,d_n(k) \Gamma_n(\mu;k).
$$

Formula \eqref{int_jack_1} was deduced in \cite{M1}  from the Kadell integral (\cite{Kad}, see  also \cite[(2.46)]{FW}),
$$\int\limits_{[0,1]^n} \!\!\widetilde C_\lambda^{\,1/k} (y)
D(y)^{\mu-\mu_0-1} D(\underline 1 - y)^{\nu-\mu_0-1} \,\omega_k(y) dy  =  \frac{d_n(k)\Gamma_n(\mu;k) \Gamma_n(\nu;k)}{\Gamma_n(\mu+\nu;k)} \cdot \frac{[\mu]_\lambda^k }{[\mu+\nu]_\lambda^k} $$
by putting $y_j = \frac{x_j}{\nu-\mu_0-1}$
and taking the limit $\nu\to\infty.$ 
We mention that in a similar way, the classical Mehta integral $c_{k,n}$
had been evaluated by Bombieri, c.f. \cite{M3}.

\begin{theorem}\label{laplace_power}  Let $\mu \in \mathbb C$ with $\,{\rm Re}\, \mu > \mu_0$ and $z \in H_n(0).$ Then
\begin{equation}\label{lap_power} \int_{\mathbb R_+^n} E_k^A(-x,z) D(x)^{\mu-\mu_0-1} \omega_k(x)dx = d_n(k)\Gamma_n(\mu;k) \cdot 
D(z)^{-\mu}, \end{equation}
where $$D(z)^a:= \prod_{j=1}^n z_j^a \quad \text{for }\, a\in \mathbb C$$
and $\zeta \mapsto \zeta^a$ denotes the principal branch of the power function on $\mathbb C \setminus ]-\infty, 0]$, satisfying
$1^a=1.$ 
\end{theorem}

Note that the Laplace integral in \eqref{lap_power} indeed converges and defines a holomorphic function on $H_n(0)$, as a consequence of the
decay properties of $E_k^A$.

\begin{proof}[Proof of Theorem \ref{laplace_power}] As $D$ and $\omega_k$ are $S_n$-invariant, it suffices to prove 
 \eqref{lap_power} with $E_k^A$ replaced by $J_k^A.$
Consider first  $z\in \mathbb C^n$ with $\|z\|_\infty < \frac{\epsilon}{n}, \,\, 0<\epsilon<1,$ and put $\alpha:= 1/k.$ 
By the factorization of $J_k^A$ according to Lemma \ref{estim2}(1), its hypergeometric expansion \eqref{power-j-a} as well as 
Lemma \ref{lemma_Macdonald}, we obtain
\begin{align*} \int_{\mathbb R_+^n} J_k^A&(-x,z+ \underline 1) D(x)^{\mu-\mu_0-1} \omega_k(x)dx  \\ 
  & = \int_{\mathbb R_+^n} \Bigl(\sum_{\lambda\in \Lambda^+_n}  
  \frac{C_\lambda^\alpha(-z)}{|\lambda|!}\,\widetilde C_\lambda^{\,\alpha}(x)\Bigr)\,e^{-\langle x, \underline 1\rangle} 
  D(x)^{\mu-\mu_0-1} \omega_k(x)dx \\
  & = \sum_{\lambda}\frac{C_\lambda^\alpha(-z)}{|\lambda|!} \int_{\mathbb R_+^n} \widetilde C_\lambda^{\,\alpha}(x)\,e^{-\langle x, \underline 1\rangle} 
  D(x)^{\mu-\mu_0-1} \omega_k(x)dx\\
  & = d_n(k) \Gamma_n(\mu;k) \cdot \sum_{\lambda}[\mu]_\lambda^k\,\frac{C_\lambda^\alpha(-z)}{|\lambda|!}\,.
\end{align*}
Here the interchange of the summation with the integral is justified by the dominated convergence theorem, since
$$ \sum_{\lambda} \frac{1}{|\lambda|!}\, |C_\lambda^\alpha(-z)\widetilde C_\lambda^{\,\alpha}(x)|\,\leq \,
\sum_{m=0}^\infty \frac{1}{m!}(n\|z\|_\infty\|x\|_\infty)^m \, = 
e^{n\|z\|_\infty\|x\|_\infty} \leq  e^{\epsilon\langle x, \underline 1\rangle}\,
$$
where estimate \eqref{jack_estim} has been used. It is known  from \cite{Ka} that for each $a\in \mathbb C$, the hypergeometric  series
$$ \sum_{\lambda\in \Lambda^+_n}[a]_\lambda^k\,\frac{C_\lambda^\alpha(z)}{|\lambda|!}\ =:\,_1F_0^\alpha(a;z) $$
converges absolutely for  $\|z\|_\infty <\rho$, provided $\rho\in ]0,1[$ is small enough.   Moreover, in this case Yan's \cite[Prop. 3.1]{Y} binomial  formula states that
\begin{equation}\label{Yan} \,_1F_0^\alpha(a;z) = D(\underline 1 - z)^{-a}.\end{equation} 
Thus for $\|z\|_\infty$ small enough, we  obtain that
$$  \int_{\mathbb R_+^n} J_k^A(-x,z+\underline 1) D(x)^{\mu-\mu_0-1} \omega_k(x)dx = d_n(k) \Gamma_n(\mu;k) D(z+\underline 1)^{-\mu},$$
and the general statement follows by analytic continuation with respect to $z$. 
\end{proof}

We now proceed to the definition of Riesz distributions associated with root systems of type $A$.
 For  a parameter $\mu \in \mathbb C$, define $\, g_\mu \in L_{\textrm{loc}}^1(\mathbb R_+^n, \omega_k)$ by 
 $$ g_\mu(x) := \frac{1}{d_n(k)\Gamma_n(\mu;k)} \cdot D(x)^{\mu-\mu_0-1}, \quad x \in \mathbb R_+^n.$$
 Theorem \ref{laplace_power} says that for $\text{Re}\,\mu > \mu_0$, the Laplace transform $\mathcal L_kg_\mu$ exists on $H_n(0)$ and is 
 given by 
 
 \begin{equation}\label{Laplace_density}
 \mathcal L_kg_\mu(z) = D(z)^{-\mu}.
\end{equation}

\begin{definition}\label{Dunkl_Riesz}
For $\mu\in \mathbb C$ with $\text{Re}\, \mu >\mu_0 $ define the 
(type $A$) Riesz measure $R_\mu \in M(\mathbb R^n)$ by  
\begin{align*}  \langle R_\mu, \varphi\rangle &:= \int_{\mathbb R_+^n}\varphi(x)g_\mu(x) \omega_k(x) dx  \\ &= \frac{1}{d_n(k)\Gamma_n(\mu;k)}\int_{\mathbb R_+^n}\varphi(x)D(x)^{\mu- \mu_0-1} \omega_k(x)dx,  
\,\, \varphi\in C_c(\mathbb R^n).\end{align*}
\end{definition}

We shall regard the complex Radon measure $R_\mu$ also as a tempered distribution on $\mathbb R^n$ with support $\overline{\mathbb R_+^n}.$ 
Notice that the mapping $\mu\to R_\mu$ is holomorphic on $\{\mu \in \mathbb C: \text{Re}\, \mu > \mu_0\}$ with values in 
$ \mathcal S^\prime(\mathbb R^n),$
i.e. $\mu \mapsto \langle R_\mu, \varphi\rangle $ is
holomorphic for each $\varphi \in \mathcal S(\mathbb R^n).$ 

The Riesz measures $R_\mu$ have already been introduced in the (unpublished) thesis \cite{L}. There also the subsequent Bernstein identity as well as Corollary \ref{main} concerning the
distributional extension of the Riesz measures with respect to $\mu$ 
were proven. For the reader's convenience, we shall nevertheless include  proofs of these results, where our proof of the Bernstein identity is slightly different from that in 
\cite{L}. 
The  Bernstein identity is based on the operator
$$ D\bigl(T(k)\bigr):= \prod_{i=1}^n T_i(k)$$
with $T_i(k)$ the $A_{n-1}$ Dunkl operators. 
Notice that $D\bigl(T(k)\bigr)$ acts as a linear differential operator of order $n$ on $C^r(\mathbb R^n)^{{\rm rad}}$, the subspace of $S_n$-invariant
functions from $C^r(\mathbb R^n)$ with $r\geq n.$

\begin{lemma}[Bernstein identity] \label{Bernstein} 
For $x\in \mathbb R_+^n$ and $a\in \mathbb C,$ 
$$ D\bigl(T(k)\bigr) D(x)^a = b_k(a) D(x)^{a-1},$$
with
$$ b_k(a) = \prod_{i=1}^n (a+k(i-1)).$$
\end{lemma}

\begin{proof} 
We claim that for $a\in \mathbb C$ and $i=1, \ldots, n$, 
\begin{equation}\label{induction}
T_i(k)\bigl(D(x)^{a-1}\cdot  x_i\cdots x_n\bigr) = \bigl(a+k(i-1)\bigr) 
D(x)^{a-1} \cdot x_{i+1} \cdots x_n \,.
\end{equation}
For the proof of this identity, note that 
for $f,g\in  C^1(\mathbb R_+^n)$ with $f$ or $g$ $S_n$-invariant, the Dunkl operators satisfy the product rule
$$ T_i(k)(fg) = T_i(k)f \cdot g + f\cdot T_i(k) g.$$
Therefore 
$$ T_i(k)\bigl(D(x)^{a-1}\cdot  x_i\cdots x_n\bigr) = 
   T_i(k)(x_i\cdots x_n) \cdot D(x)^{a-1} + \,x_i\cdots x_n \cdot \partial_i\bigl(D(x)^{a-1}\bigr).  $$
   Further, 
\begin{align*} T_i(k)(x_i \cdots x_n) &= \partial_i(x_i \cdots x_n) + k\sum_{j \not= i} 
   \frac{(x_i\cdots x_n) -\sigma_{ij}(x_i\cdots x_n)}{x_i-x_j} \notag\\ 
   &= \bigl(1+k(i-1)\bigr)\cdot x_{i+1}\cdots x_n.\end{align*}
This gives formula \eqref{induction}, from which the  statement of the Lemma follows by recursion.
\end{proof}

\begin{remark} 
An alternative, less direct proof of the Bernstein identity can be obtained from the Laplace transform identity \eqref{Laplace_density}, similar as
in \cite{FK}, Propos. VII.1.4 for symmetric cones: 
For $a \in   \mathbb C$ with $\text{Re}\, a < -\mu_0$ we have
$$ D(x)^a = \mathcal L_k g_{-a}(x)  \quad \text{ for } \,x\in \mathbb R_+^n$$
and thus by Lemma \ref{Laplace_1}, 
\begin{align*} D(T(k)) D(x)^a &=   (-1)^n\mathcal L_k(D\, g_{-a})(x) = (-1)^n \frac{\Gamma_n(1-a;k)}{\Gamma_n(-a;k)} \mathcal L_k(g_{-a+1})(x) \\
&= b_k(a) D(x)^{a-1}.
\end{align*}
For general $a \in \mathbb C$, the Bernstein identity then follows by analytic extension. 
\end{remark}

\begin{corollary}\label{main}
The mapping $\, \mu \mapsto R_\mu\,, \, \{{\rm Re}\,\mu > \mu_0\} \to \mathcal S^\prime (\mathbb R^n)$  extends uniquely to an 
$\mathcal S^\prime(\mathbb R^n)$-valued holomorphic mapping on $\mathbb C$ satisfying the recursion
$$ D\bigl(T(k)\bigr) R_\mu = R_{\mu -1} \quad (\mu \in \mathbb C).$$
\end{corollary}

\begin{proof} The proof is analogous to the case of symmetric
cones, c.f. Chapter VII of \cite{FK}. First notice that
$$ \frac{\Gamma_n(\mu;k)}{\Gamma_n(\mu-1;k)} = b_k(\mu-\mu_0-1).$$
For $\text{Re}\, \mu > \mu_0$ we extend $g_\mu$ to a locally integrable, tempered function  on $\mathbb R^n$ by putting $g_\mu:= 0$ on $\mathbb R^n \setminus \mathbb R_+^n$. 
If $\text{Re}\, \mu > \mu_0+n+1,$ then $g_\mu \in C^n(\mathbb R^n)$, and Lemma \ref{Bernstein} implies that
 $$D\bigl(T(k)\bigr)\, g_\mu = 
  g_{\mu-1}\quad \text{ on } \,\mathbb R^n.$$
From  \eqref{Dunkl_dist} and the skew-symmetry of the Dunkl operators $T_i(k)$ in $L^2(\mathbb R^n, \omega_k)$ it now follows that 
\begin{equation}\label{continuation} D\bigl(T(k)\bigr) R_\mu = R_{\mu-1}.\end{equation}
This formula recursively defines tempered distributions $R_\mu$ for all $\mu\in \mathbb C$ 
in such a way that the mapping
$ \,\mu \mapsto  R_\mu  $
is holomorphic on $\mathbb C$. The uniqueness is clear. 
\end{proof}

\begin{definition}
For a distribution $u\in \mathcal S^\prime(\mathbb R^n)$ and $\sigma\in S_n$ define $u^\sigma \in \mathcal S^\prime(\mathbb R^n)$ by
$$ \langle u^\sigma, \varphi \rangle :=  \big\langle u, \varphi^{\sigma^{-1}}\big\rangle, $$
where $\varphi^\sigma := \varphi\circ \sigma^{-1}$ for functions $\varphi: \mathbb R^n\to \mathbb C.$ 
\end{definition}

\begin{lemma}\label{remark_Riesz} The Riesz distributions $R_\mu\,, \mu \in \mathbb C\,$ 
have the following properties: 
\begin{enumerate}\itemsep=+1pt
 \item[\rm{(1)}] $R_\mu$ is 
$S_n$-invariant, i.e. $R_\mu^\sigma = R_\mu$ for all $\sigma\in S_n$. 
\item[\rm{(2)}] The support of $R_\mu$ is contained in $\overline{\mathbb R_+^n},$ i.e. $R_\mu \in \mathcal S_+^\prime(\mathbb R^n).$ 
\item[\rm{(3)}]  $\displaystyle D(x) R_\mu = \prod_{j=1}^n (\mu-k(j-1)) \cdot R_{\mu+1}.$
\end{enumerate}
\end{lemma}

All three properties are obvious for $\,{\rm Re}\,\mu >\mu_0$ and follow for 
general $\mu$ by analytic continuation.

\begin{theorem} \label{Laplace_Riesz} 
For all $\mu\in \mathbb C$  and $z\in H_n(0),$ 
\begin{equation}\label{Laplace_Rieszf} \mathcal L_k R_\mu(z) = D(z)^{-\mu}.\end{equation}
\end{theorem}

\begin{proof}  
Let $z \in H_n(0).$ We have 
$$ \mathcal L_k R_\mu(z) = \langle R_\mu, \chi E_k^A(\,.\, , -z)\rangle, $$ where 
$\chi E_k^ A(\,.\,,-z) \in \mathcal S(\mathbb R^n).$ Thus by the previous corollary, 
the mapping $\mu \mapsto \mathcal L_kR_\mu(z)$ extends analytically to $\mathbb C$. On the other hand, we already know that for $\text{Re}\, \mu > \mu_0$ and $z\in H_n(0),$
$\, \mathcal L_kR_\mu(z) = \mathcal L_kg_\mu(z) = D(z)^{-\mu}.\,$ Analytic continuation implies the assertion.
\end{proof}

\begin{corollary}\label{Riesz_null}  $R_0 = \delta_0$. 
 \end{corollary}

\begin{proof}
Theorem
 \ref{Laplace_Riesz} shows that $\mathcal L_kR_0 = 1$ on $H_n(0)$, and the statement follows by Example \ref{ex_delta_0} and the 
injectivity of the Dunkl-type Laplace transform of tempered distributions (Theorem \ref{inj_distributions}). 
\end{proof}

 In the analysis on symmetric cones, there is a famous result by Gindikin characterizing
 those Riesz distributions which are actually positive measures. Their indices are exactly those belonging to the so-called Wallach set, which plays for example an important role in the study of Hilbert spaces of holomorphic functions on symmetric domains. 

Motivated by these facts, we are now going to investigate for which indices the Dunkl type Riesz distributions $R_\mu$ are actually (positive) measures. 
We expect that in analogy to the case of symmetric cones, the distribution $R_\mu$ is a positive measure if and only if $\mu$ belongs to the generalized Wallach set
$$ W_k = \{0,k,\ldots,k(n-1)=\mu_0\}\,\cup\, ]\mu_0,\infty[\,.$$ 
We know already that $R_\mu$ is a positive measure if $\mu = 0$ or if $\mu \in \mathbb R$ with $\mu > \mu_0.$ 
In the following theorem,  we consider the Wallach points  $rk$ with $r\in \{1,2, \ldots, n-1\}.$ We start with some notation. 
For an integer $r$  with $0\leq r \leq n-1,$ we denote by $\partial_r(\mathbb R_+^n)$ the rank $r$ part of the (stratified) boundary $\partial(\mathbb R_+^n)$, which is given by 
$$\partial_r(\mathbb R_+^n) = \bigcup_{\sigma\in  S_n}\bigl\{x\in \overline{\mathbb R_+^n}\,: x_{\sigma(r+1)} = \cdots = x_{\sigma(n)} = 0\bigr\}.$$ 
Now fix $r\in \{1, \ldots, n-1\}$ and consider the factorization
$\mathbb R^n = \mathbb R^r\times \mathbb R^{n-r}$. We write $x\in \mathbb R^n$ as 
$$x= (x^\prime, x^{\prime\prime}) \text{ with }\, x^\prime = (x_1, \ldots, x_r), \, x^{\prime\prime} = (x_{r+1}, \ldots, x_n).$$
We further introduce the notations 
$$D(x^\prime) := x_1 \cdots x_r\,, \quad \omega_k(x^\prime) := \prod_{1\leq i < j \leq r} |x_i-x_j|^{2k}. $$ 
Note that $\omega_k(x^\prime) = 1$ if $r=1.$

We denote by $R_\mu^\prime$ the Dunkl-type Riesz distribution on $\mathbb R^r$ associated with root system $A_{r-1}$ and the same 
multiplicity $k>0$. Notice that here $k_0= k(r-1).$ If $r=1$, then the Dunkl setting degenerates,  and $R_\mu^\prime$ coincides with
the classical Riesz distribution (also called Riemann-Liouville distribution) on $\mathbb R$, which is defined  by
$$ \langle R_\mu^\prime, \varphi\rangle = \frac{1}{\Gamma(\mu)}\int_{\mathbb R_+} \varphi(x) x^{\mu-1} dx \quad \text{ for } \,\text{Re}\, \mu >0.$$
This distribution extends holomorphically to all $\mu \in \mathbb C$ such that $\frac{d}{dx}(R_\mu^\prime) = R_{\mu-1}^\prime $ for all 
$\mu \in \mathbb C$, c.f. \cite[Section 2.3]{FJ}.

\begin{theorem}\label{maintheorem} For $r \in \{1, \ldots, n-1\}, $ the Riesz distribution $R_{kr}$ is a positive Radon measure, namely
\begin{equation}\label{Riesz_diskret} R_{kr}  =\frac{1}{n!} \cdot \! 
\sum_{\sigma\in S_n} (R_{kn}^\prime \otimes 
\delta_0^{\prime\prime } )^\sigma,\end{equation}
where $\delta_0^{\prime\prime}$ denotes the point measure  in $0\in \mathbb R^{n-r}.$
The support of $R_{kr}$ is given by $\partial_r(\mathbb R_+^n).$ 
\end{theorem}

\begin{proof} Notice first that the Riesz distribution $R_{kn}^\prime$ on $\mathbb R^r$ is a positive measure, as $kn > k(r-1).$ Its support is  $\overline{\mathbb R_+^r},$ and therefore  the distribution on the
right side of  \eqref{Riesz_diskret} is an $S_n$-invariant positive tempered Radon measure which we denote by $m_{kr}$. It is clear that $m_{kr} \in \mathcal S_+^\prime(\mathbb R^n)$ with
$\,\textrm{supp} (m_{kr}) = \partial_r(\mathbb R_+^n).$ 
By the injectivity of the Laplace transform $\mathcal L_k$ on $\mathcal S_+^\prime(\mathbb R^n)$ (Theorem \ref{inj_distributions}) and
in view of Theorem \ref{Laplace_Riesz}, it suffices to prove that
\begin{equation}\label{id_laplace} \mathcal L_k(m_{kr})(z) = D(z)^{-kr} \quad \forall z \in H_n(0). \end{equation}
Consider first an arbitrary $S_n$-invariant test function $\varphi \in \mathcal S(\mathbb R^n).$ Then
\begin{align}\label{testformel} \langle m_{kr}, \varphi \rangle &= \frac{1}{n!}\sum_{\sigma\in S_n} \big\langle (R_{kn}^\prime \otimes \delta_0^{\prime\prime} )^\sigma, \varphi \big\rangle\,= \,  \langle R_{kn}^\prime \otimes \delta_0^{\prime\prime}, \varphi\rangle \,=\notag\\ 
&=\, \frac{1}{d_r(k)\Gamma_r(kn;k)} \int_{\mathbb R_+^r} \varphi(x^\prime, 0) \cdot 
D(x^\prime)^{kn-k(r-1)-1} \omega_k(x^\prime)dx^\prime.\end{align}
Now let $z \in H_n(0).$ As $m_{kr}$ is $S_n$-invariant, we have
$$ \mathcal L_k(m_{kr})(z) = \langle m_{kr}\,, \chi E_k^A(\,.\,,-z)\rangle =
\langle m_{kr}\,, \chi J_k^A(\,.\,,-z)\rangle$$
with some $S_n$-invariant cutoff function $\chi\in C^\infty(\mathbb R^n)$ as in the definition of the type-$A$ Laplace transform  of 
distributions. Identity \eqref{testformel} therefore 
implies that
$$ \mathcal L_k(m_{kr})(z) =  \frac{1}{d_r(k)\Gamma_r(kn;k)} \int_{\mathbb R_+^r}  J_k^A\bigl((x^\prime,0), -z\bigr) D(x^\prime)^{k(n-r+1)-1} \omega_k(x^\prime)dx^\prime.$$
The proof of Eq. \eqref{id_laplace} will now be finished by the following Lemma.
\end{proof}

\begin{lemma} Fix $r \in \{1, \ldots, n-1\}.$ Then for all $z\in H_n(0),$
$$ \frac{1}{d_r(k)\Gamma_r(kn,k)}  \int_{\mathbb R_+^r} J_k^A\bigl((x^\prime,0),-z\bigr) D(x^\prime)^{k(n-r+1)-1} \omega_k(x^\prime)dx^\prime = 
D(z)^{-kr}.$$ 
\end{lemma}

\begin{proof} As the number of variables is relevant in this proof, we shall  write 
$\, 1_n $ for the element $\underline 1 \in \mathbb R^n.$
Both sides of the stated identity are holomorphic as functions of $z\in  H_n(0),$ and therefore it  suffices to verify the identity for all arguments of the form $z + 1_n,\, \, z\in \mathbb C^n$ with $\|z\|_\infty < \epsilon$ for some $\epsilon \in \, ]0,1[.$
As in Theorem \ref{laplace_power},  we shall use the identity  $\, J_k^A(x, -z- 1_n) = J_k^A(x,-z)e^{-\langle x, 1_n\rangle}$ 
as well as the hypergeometric expansion  \eqref{power-j-a} of $J_k^A$.
The Jack polynomials in $n$ variables have the stability property 
$$  C_\lambda^\alpha(z_1, \ldots, z_r,0, \ldots, 0) = 
\begin{cases} C_\lambda^\alpha(z_1, \ldots, z_r) & \text{if }\, l(\lambda) \leq r\\
                    \,0\, & \text{otherwise},
                    \end{cases}$$
                    see \cite[Poposition 2.5]{Sta} together with \cite[formula (16)]{Ka}.
                    Hence with $\alpha = 1/k$, 
$$ J_k^A\bigl((x^\prime,0),-z\bigr) =  \sum_{\lambda \in \lambda_n^+}  \frac{C_\lambda^\alpha(x^\prime,0)\, C_\lambda^\alpha(-z)}{|\lambda|! \,C_\lambda^\alpha( 1_n)} =  
\sum_{\lambda\in \Lambda_r^+}  \frac{1}{|\lambda|!} C_\lambda^\alpha(x^\prime)\, \frac{C_\lambda^\alpha(-z)}{C_\lambda^\alpha(1_n)}.$$   
We therefore obtain  
\begin{align} \label{I(z)} I(z)&:= \int_{\mathbb R_+^r} J_k^A\bigl((x^\prime,0),-(z+1_n)\bigr) D(x^\prime)^{k(n-r+1)-1} \omega_k(x^\prime)dx^\prime \notag \\
&= \int_{\mathbb R_+^r} \Bigl(\sum_{\lambda\in \Lambda_r^+}  
  \frac{1}{|\lambda|!}\frac{ C_\lambda^\alpha(-z)}{C_\lambda^\alpha(1_n)}\,C_\lambda^\alpha(x^\prime)\Bigr)e^{-\langle x^\prime, 1_r\rangle} 
  D(x^\prime)^{k(n-r+1)-1} \omega_k(x^\prime)dx^\prime \notag \\
  & = \sum_{\lambda\in \Lambda_r^+}\frac{C_\lambda^\alpha(-z)}{|\lambda|!} 
  \cdot \frac{C_\lambda^\alpha(1_r)}{C_\lambda^\alpha( 1_n)}
  \int_{\mathbb R_+^r} \widetilde C_\lambda^\alpha(x^\prime) e^{-\langle x^\prime,1_r\rangle}
  D(x^\prime)^{k(n-r+1)-1} \omega_k(x^\prime)dx^\prime.
  \end{align}
Here the interchange of the summation and integral is justified by the dominated convergence theorem for  
$\|z\|_\infty< \epsilon  $ with $\epsilon <\frac{1}{r}\,$, because by  \eqref{jack_hilf},
$$ \sum_{\lambda\in \Lambda_r^+} \frac{1}{|\lambda|!}\big\vert 
  \widetilde C_\lambda^\alpha(-z)\,C_\lambda^\alpha(x^\prime)\big\vert \,\leq \sum_{m=0}^\infty \frac{1}{m!}(r\|z\|_\infty \|x^\prime\|_\infty)^m  \,\leq \, e^{r\|z\|_\infty\|x^\prime\|_1}\,.$$
  For $\lambda \in \Lambda_n^+$ we further have  
  $$ \frac{C_\lambda^\alpha( 1_r)}{C_\lambda^\alpha(1_n)} = \frac{[kr]_\lambda^k}{[kn]_\lambda^k}\,,$$
  see e.g. \cite[formula (17)]{Ka}.    Inserting this in \eqref{I(z)} and evaluating  the integral by means of formula \eqref{int_jack_1}, we 
  obtain (for $\|z\|_\infty $ small enough) that
 \begin{equation}\label{I_2} I(z) = d_r(k)\Gamma_r(kn;k) \sum_{\lambda\in \Lambda_r^+}[kr]_\lambda^k \cdot 
  \frac{C_\lambda^\alpha(-z)}{|\lambda|!}.\end{equation}
  On the other hand, Yan's binomial formula \eqref{Yan} gives, for small 
  $\|z\|_\infty$, 
  $$ D(z + 1_n)^{-kr} = \, _1F_0^\alpha(kr;-z) = \sum_{\lambda\in \Lambda_n^+} [kr]_\lambda^k \cdot \frac{C_\lambda^\alpha(-z)}{|\lambda|!}  .$$
If $l(\lambda) > r,$ then 
 $$ [kr]_\lambda^k = \prod_{j=1}^{l(\lambda)} \bigl(k(r-j+1)\bigr)_{\lambda_j}\, =0. 
 $$ 
Therefore
 $$ D(z + 1_n)^{-kr} = \sum_{\lambda\in \Lambda_r ^+}[kr]_\lambda^k \cdot 
  \frac{C_\lambda^\alpha(-z)}{|\lambda|!}.$$
This shows that for $\|z\|_\infty$ sufficiently small,
$$ I(z) =  d_r(k)\Gamma_r(kn;k) \cdot D(z+ 1_n)^{-kr}, $$
which finishes the proof of the lemma.
  \end{proof}

We mention that some partial results on the  distributions $R_{kr}$ are sketched  in \cite{L}.

We are now aiming at necessary conditions under which the Riesz distribution $R_\mu$ is a complex or even a positive measure. 
We know already that $R_\mu$ is a positive measure if $\mu$ belongs to the generalized Wallach set $W_k$, and that it is a complex (non-positive) measure 
 if $\mu \in \mathbb C\setminus \mathbb R$  with $\text{Re}\, \mu > \mu_0 = k(n-1).$ 
We shall use 
methods of Sokal \cite{S} for Riesz distributions on symmetric cones, 
as well as  the
following variant of a principle due to Shanbhag, Casalis and Letac (see \cite{CL} as well as \cite{S}):

\begin{lemma}[Shanbhag-Casalis-Letac prinicple in the Dunkl setting] \label{Shanbhag} Suppose that $m\in \mathcal S_+^\prime(\mathbb R^n) $ 
is a positive tempered Radon measure and $p\in \mathbb C[\mathbb R^n]$ is a polynomial which is real-valued and 
non-negative on \,{\rm supp}\,m, then
$$p\big(-T(k)\big)\mathcal L_km \geq 0  \text{ on }\,\mathbb R_+^n.$$   
\end{lemma}

\begin{proof} In view of Remark \ref{Laplace_measure} we have $\mathcal L_km \in \mathbb C^\infty(\mathbb R_+^n)$ and may differentiate under 
the integral. Thus for $x \in \mathbb R_+^n$,
$$ p\big(-T(k)\big) \mathcal L_km(x) = 
\int_{\overline{\mathbb R_+^n}} p(y) E_k^A(y, -x) dm(y)\, \geq 0.$$
\end{proof}

Recall from Theorem \ref{Laplace_Riesz} that for each $\mu \in \mathbb C$, 
$$ \mathcal L_k R_\mu(x) = D(x)^{-\mu} \quad \text{on } \,\mathbb R_+^n.$$
The following evalutation formula involving Dunkl operators associated with the (renormalized) Jack polynomials will be important lateron; it is an analogue of the formula 
on top of p. 245 in \cite{FK} for the spherical functions on  symmetric cones.

\begin{lemma}\label{eval_lemma} Let $\alpha = 1/k.$ Then 
for $\mu \in \mathbb C$ and each partition $\lambda$ of length $l(\lambda)\leq n$,
\begin{equation}\label{eval}  \widetilde C_\lambda^\alpha(T(k)) D(\underline 1-x)^{-\mu} \big\vert_{x=0} =\, [\mu]_\lambda^k\,.\end{equation}
 
\end{lemma}

\begin{proof} Suppose first that $\text{Re}\, \mu > \mu_0.$ Then according to Theorem \ref{laplace_power}, 
\begin{align*}  D(\underline 1-x)^{-\mu} \, & = \, \frac{1}{d_n(k)\Gamma_n(\mu;k)} \int_{\mathbb R_+^n}  E_k^A(-y,\underline 1-x) D(y)^{\mu-\mu_0-1}\omega_k(y)dy \\
&=  \, \frac{1}{d_n(k)\Gamma_n(\mu;k)} \int_{\mathbb R_+^n}  E_k^A(y,x) e^{-\langle y, \underline 1\,\rangle} D(y)^{\mu-\mu_0-1}\omega_k(y)dy.
\end{align*}
Differentiating under the integral gives
$$ \widetilde C_\lambda^\alpha(T(k)) D(\underline 1-x)^{-\mu} = \frac{1}{d_n(k)\Gamma_n(\mu;k)}\! \int_{\mathbb R_+^n} \!\!\widetilde C_\lambda^\alpha (y)E_k^A(y,x) 
e^{-\langle y, \underline 1\rangle} 
D(y)^{\mu-\mu_0-1}\omega_k(y)dy.$$ 
Thus by Lemma \ref{lemma_Macdonald},
\begin{align*} \widetilde C_\lambda^\alpha(T(k)) D(\underline 1-x)^{-\mu} \big\vert_{x=0} & = \frac{1}{d_n(k)\Gamma_n(\mu;k)} \int_{\mathbb R_+^n} 
\widetilde C_\lambda^\alpha (y) \,e^{-\langle y, \underline 1\rangle} 
D(y)^{\mu-\mu_0-1}\omega_k(y)dy\\
& = \,[\mu]_\lambda^k.\end{align*}
Both sides of formula \eqref{eval} are holomorphic in $\mu \in \mathbb C$ (for the left side note that $\widetilde C_\lambda^\alpha(T(k))$ acts as
a differential operator on  $S_n$-invariant functions), and therefore the stated identity extends to all $\mu \in \mathbb C.$
\end{proof}

\begin{theorem}\label{maintheorem2} Consider the Riesz distributions $R_\mu\,, \,\,\mu \in \mathbb C. $
\begin{enumerate} 
\item[\rm{(1)}] If $R_\mu$ is a complex measure, then either\enskip  ${\rm Re}\, \mu > \mu_0 = k(n-1)$, 
or $\mu$ is contained in the finite set  
 $$[0,\infty[\,\cap \,\bigl(\{0, k, \ldots k(n-1)\}- \mathbb N_0\bigr).$$ 
 \item[\rm{(2)}] $R_\mu$ is a positive measure if and only if 
  $\mu$ is contained in the generalized Wallach set $\,W_k = \{0,k, \ldots, k(n-1)\}\,\cup\,]k(n-1), \infty[.$ 
 \end{enumerate}
\end{theorem}

\begin{proof} (1) We shall apply Proposition 2.3 of  \cite{S}. For this, consider the regular distributions 
$u_\mu \in \mathcal D^\prime(\mathbb R_+^n), \,\, \mu \in \mathbb C,$  which are defined by the densities
\begin{equation}\label{Riesz_densities}  f_\mu(x)  = \frac{1}{d_n(k)\Gamma_n(\mu;k)} \cdot D(x)^{\mu-\mu_0-1} \omega_k(x) \, 
\in L^1_{\textrm{loc}}(\mathbb R_+^n),  \end{equation} 
that is $$\, \langle u_\mu, \varphi\rangle = \int_{\mathbb R_+^n} \varphi(x) f_\mu(x)dx, \quad \varphi\in \mathcal D(\mathbb R_+^n).$$
Notice that $\mu \mapsto f_\mu(x)$ is holomorphic on $\mathbb C$ for each 
$x\in \mathbb R_+^n$ and that $\mu \mapsto u_\mu$ is holomorphic on $\mathbb C$ with values in $\mathcal D^\prime(\mathbb R_+^n).$
If $\text{Re}\, \mu > \mu_0,$ then the Riesz distribution $R_\mu$  provides an extension of the distribution $u_\mu$ to a 
distribution on $\mathbb R^n$ in the sense that the restriction of $R_\mu$ to $\mathbb R_+^n$ coincides with $u_\mu\,.$ Moreover, in this case 
$R_\mu$ is 
also a complex measure, i.e. of order zero. In addition we know that the mapping $\mu \mapsto R_\mu$ is holomorphic on all of $\mathbb C$.
We are therefore in the situation of  \cite[Proposition 2.3]{S}, which yields the following conclusions: First, 
 $R_\mu$ extends $u_\mu$ for each $\mu \in \mathbb C$, and second, if $R_\mu$ is a complex measure on $\mathbb R^n$,
then the density $f_\mu$ (extended by zero to all of $\mathbb R^n$) must belong to $L_{\textrm{loc}}^1(\mathbb R^n)$. But this implies that either
$\text{Re}\, \mu > \mu_0$ or $\mu$ is a pole of  $\Gamma_n(\,.\,; k),$ i.e. $\mu \in \{0, k, \ldots k(n-1)\}- \mathbb N_0\,.$ (Note that in the 
latter case, $f_\mu$ is identical zero.)
If $\mu \in ]-\infty,0[$, then it follows from the Laplace transform formula \eqref{Laplace_Rieszf} that $R_\mu$ cannot be a  measure. 
Indeed, suppose that $R_\mu$ is a (tempered) measure. Then for $x\in \mathbb R$ with $x>0$ we have 
$$ \mathcal L_kR_\mu(\underline x) = \int_{\overline{\mathbb R_+^n}} E_k^A(y, -\underline x) dR_\mu(y) = \int_{\overline{\mathbb R_+^n}} e^{-\langle y, \underline x\rangle} dR_\mu(y),$$
which is a usual Laplace transform. 
Thus by  Lemma 3.6. of \cite{S}, $x\mapsto \mathcal L_kR_\mu(\underline a + \underline x)$ is bounded on $[0, \infty[$ for each $a>0.$ But on the other hand, $\mathcal L_kR_\mu(\underline a + \underline x) = 
(a + x)^{-n\mu}$, 
which is unbounded as $x\to \infty.$ 

\smallskip

 (2) In view of Corollary \ref{Riesz_null} and Theorem  \ref{maintheorem} it remains to prove the ``only if'' part. 
 Suppose that $R_\mu$ is a positive measure. We have to exclude the possibility that $\mu$ belongs to one of the 
 open intervalls $\,]k(r-1),kr[\,$ with 
 $ r\in \{1, \ldots, n-1\}.$ 
For this, we  apply the Shanbhag-Casalis-Letac principle to $R_\mu$ and the Jack polynomials $\widetilde C_\lambda^\alpha$, which are non-negative on 
$\overline{\mathbb R_+^n}$ as a consequence of their non-negative monomial expansion \eqref{Jack_pos}. 
Thus by Lemma \ref{Shanbhag}, 
$$\widetilde C_\lambda^\alpha(-T(k)) (\mathcal L_k R_\mu) \geq 0 \quad \text{ on }\, \mathbb R_+^n.$$
Employing Lemma \ref{eval_lemma} we therefore obtain that for all $\lambda \in \Lambda^+_n$,
\begin{equation}\label{Shanbhag2}  [\mu]_\lambda^k = \widetilde C_\lambda^\alpha(T(k)) D(\underline 1-x)^{-\mu}\vert_{x=0} 
= \bigl(\widetilde C_\lambda^\alpha(-T(k))\mathcal L_kR_\mu\bigr) (\underline 1) \geq 0.\end{equation}
Here for the second equality, it was used that for   $f\in C^1(\mathbb R^n)$ and $g(x) := f(\underline 1-x), $ the Dunkl operators satisfy 
$\, (T_i(k)f)(\underline 1-x) = -T_i(k)g(x).$
Now suppose that $\mu \in ]k(r-1), kr[\,, \,\,r\in \{1, \ldots, n-1\}.$  Choose $\lambda := (1, \dots, 1, 0, \ldots)\in \Lambda_n^+$ with exactly $r+1$ parts 
equal to $1$. Then
$$ [\mu]_\lambda^k = \prod_{j=1}^{r+1}(\mu-k(j-1)) \, < 0,$$
in contradiction to \eqref{Shanbhag2}.
\end{proof}

\begin{remark} We mention that part (2) can be proven directly without referring to part (1) and Proposition 2.3 of \cite{S}. Indeed, 
 if $R_\mu$ is a positive measure, then $\mathcal L_kR_\mu$ is non-negative on $\mathbb R_+^n$ and thus by Theorem \ref{Laplace_Riesz}, 
 $\mu$ must be real-valued. Then it remains to exclude the intervals $\,]k(r-1),kr[\,$ as above.
 
\end{remark}

\section*{Acknowledgement} It is a pleasure to thank Michael Voit for many helpful discussions.


\end{document}